\documentclass{amsart}
\usepackage{latexsym, url}
\usepackage{hyperref}
\usepackage[utf8]{inputenc}
\usepackage{amsthm}
\usepackage{amsmath}
\usepackage{amsfonts}
\usepackage{amssymb}
\usepackage[final]{showkeys}    
\usepackage[dvips]{graphicx}
\usepackage{xypic}
\input xy
\xyoption{all}

\newtheorem{theo}{Theorem}[section]
\newtheorem{thm}[theo]{Theorem}
\newtheorem{lemma}[theo]{Lemma}

\newtheorem{fact}[theo]{Fact}
\newtheorem{nota}[theo]{Notation}

\newtheorem{propo}[theo]{Proposition}
\newtheorem{defi}[theo]{Definition}

\newtheorem{cor}[theo]{Corollary}
\newtheorem{rem}[theo]{Remark}
\newtheorem{remark}[theo]{Remark}

\newcommand\lsnk{LS(\ck)}
\newcommand\slesseq{\trianglelefteq}

\newcommand\sless{\triangleleft}
\newcommand\sgreat{\triangleright}
\newcommand\sgeq{\trianglerighteq}

\newcommand\Mod{\operatorname{Mod}}

\newcommand\Homk{\operatorname{Hom}_\ck}

\newcommand\Set{\operatorname{\bf Set}}

\newcommand\Str{\operatorname{\bf Str}}

\newcommand\cof{\operatorname{cf}}

\newcommand\cu{\mathcal {U}}

\newcommand\ck{\mathcal {K}}
\newcommand\cl{\mathcal {L}}
\newcommand\cm{\mathcal {M}}

\newcommand\cp{\mathcal {P}}

\newcommand{\cf}[1]{\text{cf}{(#1)}}
\newcommand{\pres}[2]{\operatorname{\bf Pres}_{#1}(#2)}
\newcommand{\Red}[3]{\operatorname{\bf Red}_{#1,#2}(#3)}

 \newbox\noforkbox \newdimen\forklinewidth
\forklinewidth=0.3pt \setbox0\hbox{$\textstyle\smile$}
\setbox1\hbox to \wd0{\hfil\vrule width \forklinewidth depth-2pt
 height 10pt \hfil}
\wd1=0 cm \setbox\noforkbox\hbox{\lower 2pt\box1\lower
2pt\box0\relax}

\title[A category-theoretic characterization of almost measurable cardinals]
      {A category-theoretic characterization of almost measurable cardinals}
\date{\today
}

\author[M. Lieberman]{Michael Lieberman}
\email{lieberman@math.muni.cz}
\urladdr{http://www.math.muni.cz/\textasciitilde lieberman/}
\address{Department of Mathematics and Statistics, Faculty of Science, Masaryk University, Brno, Czech Republic}
\address{Department of Mathematics, Faculty of Mechanical Engineering, Brno University of Technology, Brno, Czech Republic}

\thanks{The author is supported by the Grant Agency of the Czech Republic under the grant P201/12/G028.}

\begin{document}

\begin{abstract} Through careful analysis of an argument of \cite{BT-R}, we show that the powerful image of any accessible functor is closed under colimits of $\kappa$-chains, $\kappa$ a sufficiently large almost measurable cardinal.  This condition on powerful images, by methods resembling those of \cite{LRclass}, implies $\kappa$-locality of Galois types.  As this, in turn, implies sufficient measurability of $\kappa$, via \cite{boun}, we obtain an equivalence: a purely category-theoretic characterization of almost measurable cardinals.
\end{abstract} 

\maketitle


\section{Introduction}\label{secintro}

This paper fits into the emerging stream of research on the interactions between large cardinals, abstract model theory, and accessible categories.  The aim of this broader project is, in short, to find natural model-theoretic and category-theoretic equivalents of large cardinal principles.  An extraordinary early success along these lines is the following result of Boney and Unger:

\begin{thm}[\cite{boun} 4.14]\label{thmboun}
	Let $\kappa$ be an infinite (regular) cardinal such that $\mu^\omega<\kappa$ for all $\mu<\kappa$.  The following are equivalent:
	\begin{enumerate}
		\item $\kappa$ is almost strongly compact.
		\item The powerful image of any accessible functor $F:\ck\to\cl$, $\cl$ essentially below $\kappa$, is $\kappa$-accessible.
		\item Every {\em abstract elementary class} (or {\em AEC}) essentially below $\kappa$ is $<\kappa$-tame.
	\end{enumerate}
\end{thm}

The meaning of ``essentially below'' in each case will become clear in the sequel.  In unparametrized form, the theorem above implies that the existence of a proper class of almost strongly compact cardinals is equivalent to the tameness of AECs, and to the accessibility of powerful images of accessible functors.  The most difficult step of this equivalence, (3) implies (1), is proven in \cite{boun} via an improved version of the construction and proof in \cite{shelahmeas}.  The implication from (2) to (3) follows from work in \cite{LRclass}, in which tameness is reformulated as a question involving powerful images.  That accessibility of powerful images follows from strongly compact cardinals has long been known (see \cite{makkai-pare}), while the weakening to almost strongly compacts is due to \cite{BT-R}---the product, in essence, of a careful rereading of \cite{makkai-pare}.

We wish, instead, to obtain a similar result for weakenings of the notion of {\em measurable cardinal}.  We note that such weakenings have been considered to some degree in connection with AECs and, in particular, with locality---rather than tameness---of Galois types therein (cf. \cite{boun}).  Relatively little is known about how the properties of accessible categories change under the assumption of the existence (or nonexistence) of measurable cardinals, let alone their weakenings.  The current state of knowledge seems not to extend beyond the following:

\begin{fact}
	\begin{enumerate}
		\item If every accessible category is co-well-powered, i.e. every object has a set of quotients, then there is a proper class of measurable cardinals, \cite[A.19]{adamek-rosicky}.
		\item The dual category of the category of sets, $\Set^{op}$, is bounded if and only if there are boundedly many measurable cardinals, \cite[A.5]{adamek-rosicky}.
	\end{enumerate}
\end{fact}

It is not known whether the converse to (1) holds, and (2) is elegant, but appears to be of limited practical applicability.  We aim to refine this picture by exhibiting a useful and relatively straightforward category-theoretic property which is equivalent to (almost) measurability of a single cardinal.  In particular, the culmination of this piece is an equivalence paralleling Theorem~\ref{thmboun} above, linking almost measurable cardinals, locality of Galois types in AECs, and the closure of powerful images of accessible functors under colimits of suitable chains.  

As we work in the intersection of several technical fields, we assume the reader has a certain amount of familiarity with abstract model theory in the form of AECs or $\mu$-AECs (cf. \cite{baldwinbk} and \cite{mu-aec-jpaa}, respectively), accessible categories (cf. \cite{adamek-rosicky}), and the rudiments of the theory surrounding large cardinals (c.f. \cite{jechbook}).  We include only a brief review of the relevant terminology, in Section~\ref{secprelims} below.

In Section~\ref{secaccims}, we modify the argument of \cite{BT-R} so that it applies in case $\kappa$ is a $\mu$-measurable cardinal (Definition~\ref{defmeasurable}(1)).  From this we are able to derive that powerful images of $\lambda$-accessible functors below $\mu$ are $\kappa$-preaccessible and closed under colimits of $\kappa$-chains (Theorem~\ref{thmmeasmono}).  In Section~\ref{seclocal}, this closure under colimits of $\kappa$-chains is shown to imply $\kappa$-locality of AECs below $\mu$ (Theorem~\ref{thmmuaec}), by an argument closely resembling that of \cite{LRclass}.  Finally, we rely on a result of \cite{boun} to derive $\mu$-measurability of $\kappa$ from $\kappa$-locality, thereby completing the equivalence and providing the promised category-theoretic characterization of almost measurable cardinals: Theorem~\ref{almeasequiv}.  We note in passing that the distillation of the arguments of \cite{BT-R} and \cite{LRclass} employed here has since proved useful, too, in deriving level-by-level category-theoretic characterizations---again, given in terms of closure properties of powerful images---of cardinals in the range from weakly to strongly compact, see \cite{meandwill}.

The author wishes to acknowledge useful conversations with Will Boney, Ji\v r\'i Rosick\'y, and Sebastien Vasey, and to thank Andrew Brooke-Taylor for organizing the workshop ``Accessible categories and their connections'' at the University of Leeds, where the idea for this paper began to germinate.

\section{Preliminaries}\label{secprelims}

We concern ourselves with the following weakenings of the notion of measurable cardinal, found in \cite{boun}:

\begin{defi}\label{defmeasurable}
	Let $\kappa$ be an uncountable cardinal.
	\begin{enumerate}
		\item We say that $\kappa$ is {\em $\mu$-measurable} if there is a uniform $\mu$-complete ultrafilter on $\kappa$.
		\item We say that $\kappa$ is {\em almost measurable} if it is $\mu$-measurable for all $\mu<\kappa$.
		\item We say that $\kappa$ is {\em measurable} if it is $\kappa$-measurable.
	\end{enumerate}
\end{defi}

\begin{remark}\label{unifnonsing}
{\em The assumption that the ultrafilter is uniform, and not simply non-principal, is rather strong.  Building uniformity into the definition ensures that an almost measurable cardinal is either measurable or a regular limit of measurable cardinals (hence a strong limit cardinal), or, as suggested by the anonymous referee, the successor of a measurable cardinal.  If we were to assume merely that the ultrafilter is nonprincipal, an almost measurable cardinal could conceivably be a {\em singular} limit of measurable cardinals. We wish to avoid this possibility.}
\end{remark}

\begin{lemma}\label{uppersets}
	Let $\cu$ be a uniform ultrafilter on $\kappa$.  For any $i<\kappa$, the upper set
	$$[i,\kappa)=\{j\geq i\,|\,j<\kappa\}$$
	belongs to $\cu$.
\end{lemma}

\begin{proof}
	For any $i<\kappa$, $|[0,i)|<\kappa$, hence, by uniformity of $\cu$, $[0,i)\not\in\cu$.  So
	$$[i,\kappa)=\kappa\setminus [0,i)\in\cu.$$ 
\end{proof}

\begin{lemma}\label{cofmumeas}
If $\kappa$ is $\mu$-measurable, $\cf{\kappa}\geq\mu$.	
\end{lemma}

\begin{proof}
	Suppose that $\cf{\kappa}=\chi<\mu$, and let $\{i_\alpha\,|\,\alpha<\chi\}$ be a cofinal sequence in $\kappa$.  Let $\cu$ be a uniform, $\mu$-complete ultrafilter on $\kappa$.  By Lemma~\ref{uppersets}, $[i_\alpha,\kappa)\in\cu$ for all $\alpha$.  Since $\cu$ is $\mu$-complete, it contains
	$$\bigcap_{\alpha<\chi}[i_\alpha,\kappa).$$
	But, of course, this intersection is empty.
\end{proof}

\begin{cor}\label{almostmeasreg}
	Any almost measurable cardinal is regular.
\end{cor}

As we proceed by a careful analysis of the central proof of \cite{BT-R}, we also recall the cardinal notions at play there.  Notice, though, that we refer to these notions using the terminology of \cite[4.6]{BaMa}, as it is more in keeping with Definition~\ref{defmeasurable} (and, of course, \cite{boun}).

\begin{defi}\label{defcompact}
	Let $\kappa$ be an uncountable cardinal.
	\begin{enumerate}
		\item We say that $\kappa$ is {\em $\mu$-strongly compact} if any $\kappa$-complete filter on a set $I$ can be extended to a $\mu$-complete ultrafilter.  (In \cite[2.3]{BT-R}, such cardinals are called {\em $L_{\mu,\omega}$-compact}.)
		\item We say that $\kappa$ is {\em almost strongly compact} if it is $\mu$-strongly compact for all $\mu<\kappa$.
		\item We say that $\kappa$ is {\em strongly compact} if it is $\kappa$-strongly compact.
	\end{enumerate}
\end{defi}

We note that, given the recent proliferation of possible interpretations of ``$\mu$-strongly compact,'' there is some risk of confusion: we fix the interpretation above, and trust that this will not present a serious obstacle for the reader.

Finally, we briefly review the required terminology from category theory and abstract model theory.

\begin{defi}
	Let $\lambda$ be a regular cardinal.
	\begin{enumerate}
	\item An object $M$ in a category $\ck$ is {\em $\lambda$-presentable} if the associated hom-functor $\Homk(M,-)$ preserves $\lambda$-directed colimits.
	\item We say that a category $\ck$ is {\em $\lambda$-accessible} if it
		\begin{itemize}
			\item contains a set (up to isomorphism) of $\lambda$-presentable objects,
			\item every object of $\ck$ is a $\lambda$-directed colimit of $\lambda$-presentable objects, and
			\item $\ck$ has all $\lambda$-directed colimits.
		\end{itemize}
		\item We say that $\ck$ is $\lambda$-preaccessible if it satisfies only the first two conditions in (2) above.
		\item A functor $F:\ck\to\cl$ is {$\lambda$-accessible} if $\ck$ and $\cl$ are $\lambda$-accessible and $F$ preserves $\lambda$-directed colimits.
	\end{enumerate}
\end{defi}

\begin{nota}\label{pressubcat}
	Given a category $\cl$, we denote by $\pres{\lambda}{\cl}$ a full subcategory of $\cl$ containing exactly one object from each isomorphism class of $\lambda$-presentable objects in $\cl$.
\end{nota}

\begin{remark}\label{rmksharp}
	{\em A $\mu$-accessible category need not be $\kappa$-accessible for all regular $\kappa>\lambda$.  This is true, however, when $\kappa$ is {\it sharply greater than} $\mu$, denoted $\kappa\sgreat\mu$, a relation first defined in \cite[2.3.1]{makkai-pare}.  We refer the reader to \cite{makkai-pare} for a number of conditions equivalent to $\kappa\sgreat\mu$.  At a minimum, we note the following: if $\kappa>2^{<\mu}$, $\kappa\sgreat\mu$ just in case $\kappa$ is $\mu$-closed; that is, $\beta^{<\mu}<\kappa$ for all $\beta<\kappa$ (see \cite[2.6]{internal-sizes-v2}).} 
\end{remark}

As an immediate application, which will be of use later, we note that:

\begin{propo}\label{almeassharp}
	If $\kappa$ is an almost measurable cardinal, then for any regular $\mu<\kappa$, $\mu\sless\kappa$.
\end{propo}

\begin{proof}
	From Corollary~\ref{almostmeasreg}, $\kappa$ is regular. If $\kappa$ is measurable or a regular limit of measurables, one can easily see, using Remark~\ref{unifnonsing}, that $\kappa$ must be strong limit, i.e. $2^\nu<\kappa$ for all $\nu<\kappa$.  Let $\mu<\kappa$, $\mu$ regular.  For any $\alpha<\mu$ and $\beta<\kappa$, 
	$$\beta^\alpha\leq\max(2^\alpha,2^{\beta})<\kappa.$$
	So, by the final clause of Remark~\ref{rmksharp}, $\mu\sless\kappa$.  If $\kappa$ is the successor of a measurable cardinal, say $\kappa=\theta^+$, we have $\mu\sless\theta$ for all regular $\mu<\theta$, from the argument above.  Since $\theta\sless\theta^+=\kappa$ (see \cite[2.3.3]{makkai-pare}) and the relation $\sless$ is transitive, the result follows in this case, too.
\end{proof}

We will be concerned here with the properties of the images of accessible functors.  Chiefly,

\begin{defi}
	Given a functor $F:\ck\to\cl$, the powerful image of $\ck$, denoted $P(F)$, is the least full subcategory of $\cl$ containing $FA$ for all $A\in\ck$ and closed under $\cl$-subobjects.
\end{defi}

For the purposes of Section~\ref{secaccims}, we will also need a refinement of the notion of monomorphism, namely $\lambda$-pure morphisms, which are closely connected to, for example, the notion of pure subgroup.

\begin{defi}\label{defpure} Let $\lambda$ be a regular cardinal. A morphism $f:M\to N$ is {\em $\lambda$-pure} if for every commutative square
		$$  
      	\xymatrix@=3pc{
        	M \ar[r]^{f} & N \\
        	K \ar [u]^{u} \ar [r]_{g} &
        	L \ar[u]_{v}
      	}
    	$$
    	where $K$ and $L$ are $\lambda$-presentable, there is a morphism $t:L\to M$ with $tg=u$.
\end{defi}

\begin{rem}
	In general, a $\lambda$-pure morphism need not be a monomorphism.  By \cite[2.29]{adamek-rosicky}, though, $\lambda$-pure morphisms are monomorphisms in any $\lambda$-accessible category.
\end{rem}

\begin{defi}
	Given a functor $F:\ck\to\cl$, the {\em $\lambda$-pure powerful image of $\ck$}, denoted $P_\lambda(F)$, is the least full subcategory of $\cl$ containing $FA$ for all $A\in\ck$ and closed under $\lambda$-pure $\cl$-subobjects.
\end{defi}

The impulse that led to the development of accessible categories---namely, the dream of an abstract characterization of classes of models---manifested itself on the model-theoretic side as well, in the form of Shelah's AECs (cf. \cite{shelahaecs}) and, more recently, the $\mu$-AECs of \cite{mu-aec-jpaa}.  We will not define these classes in detail---for most purposes, it suffices to think of them simply as category-theoretic generalizations of elementary classes of structures with elementary embeddings---but we note the following, directing curious readers to \cite{baldwinbk} and \cite{mu-aec-jpaa} for more information:

\begin{fact}\label{aecmuaeccats}{\em 
	\begin{enumerate}
		\item Let $\ck$ be an AEC with L\"owenheim-Skolem-Tarski number $\lsnk=\lambda$. If we consider $\ck$ as a category with $\ck$-embeddings as morphisms, it is $\lambda^+$-accessible, has arbitrary {\em directed} colimits, and all morphisms are monomorphisms (\cite[4.8]{catpap},\cite[5.5]{rosbek}). Take $U:\ck\to\Set$ to be the underlying set functor.  Then, moreover, the directed colimits and monomorphisms of $\ck$ are concrete; that is, they are preserved by $U$.  Moreover, a category satisfying the above properties is nearly an AEC (see \cite{rosbek}).
		
		We note, too, that in this case an object is $\mu$-presentable, $\mu>\lambda$, just in case $|UM|<\mu$.  This means that, by an easy counting argument, 
		$$|\pres{\mu}{\ck}|\leq 2^{<\mu},$$ 
		hence it is certainly smaller than the next strong limit cardinal.  We will need this fact in the proof of Theorem~\ref{almeasequiv} below.
		\item Let $\ck$ be a $\mu$-AEC with L\"owenheim-Skolem-Tarski number $\lsnk=\theta$. Then, considered as a concrete category $(\ck,U)$ as above, $\ck$ is $\theta^+$-accessible, with all {\em $\mu$-directed} colimits and all morphisms monomorphisms.  In this case, $\mu$-directed colimits and monomorphisms in $\ck$ are concrete (\cite[4.3]{mu-aec-jpaa}).
		
		In this case, we can specify precisely what categories are $\mu$-AECs for some $\mu$: if $\ck$ is a $\lambda$-accessible category with all morphisms monomorphisms, then, by \cite[4.10]{mu-aec-jpaa}, it is equivalent to a $\lambda$-AEC with $\lsnk=(\lambda+|\pres{\lambda}{\ck}|)^{<\lambda}$.  So, in fact, we may shift freely between the framework of $\mu$-AECs and that of accessible categories with all morphisms monomorphisms.
		
		The connection between presentability and cardinality is more delicate in $\mu$-AECs, but \cite[4.12]{internal-sizes-v2} ensures that if an object $M$ in a $\mu$-AEC is $\lambda$-presentable, $\lambda$ larger than the L\"owenheim-Skolem-Tarski number (hence also larger than $\mu$), then the cardinality of the underlying set is no larger than ${(\lambda^-)}^{<\mu}$, where
		$$\lambda^{-}=\left\{\begin{array}{ccc}\theta & \hspace{2 mm} & \mbox{if }\lambda=\theta^+\vspace{1 mm}\\ \lambda & & \mbox{if }\lambda\mbox{ limit}\end{array}\right..$$
		Counting possible $\mu$-ary structures on sets of cardinality $(\lambda^-)^{<\mu}$, one can see, at the very least, that $|\pres{\lambda}{\ck}|$ will certainly be less than the first inaccessible cardinal larger than $\lambda$.
	\end{enumerate}	}
\end{fact}

\section{Accessible images revisited}\label{secaccims}

Our goal in this section is to show that, under the assumption of a sufficiently measurable cardinal $\kappa$, the powerful image of an accessible functor is closed under colimits of $\kappa$-chains.  First, we make explicit the kind of chains we are willing to entertain:

\begin{defi} Let $\kappa$ be an infinite cardinal, $\cl$ an arbitrary category.
\begin{enumerate}\item Given a chain $D:\kappa\to\cl$, we say that a subchain $D':S\to\cl$ is \emph{cofinal in $D$} if $D$ and $D'$ have the same colimit in $\cl$(provided, of course, that the colimit exists).
	\item We say that a chain $D:\kappa\to\cl$ is {\em nondegenerate} if any cofinal subchain $D':S\to\cl$ has $|S|\geq\cof(\kappa)$.  That is, we allow only such degeneracies as arise from the set-theoretic properties of $\kappa$ itself.		
\end{enumerate}
\end{defi}

Henceforth, we assume that all our chains are nondegenerate.

In what follows, we will make heavy use of a particular notion of complexity associated with an accessible category $\cl$, namely the parameter $\mu_\cl$ of \cite[3.1]{BT-R}.  For the sake of completeness, we include the definition here:

\begin{defi}\label{sharpdef}
	Let $\cl$ be a $\lambda$-accessible category.  Let $\gamma_{\cl}$ denote the smallest cardinal with $\gamma_{\cl}\sgeq\lambda$ and $\gamma_{\cl}\geq|\pres{\lambda}{\cl}|$: as noted in \cite{BT-R}, $\gamma_\cl\leq (2^{|\pres{\lambda}{\cl}|})^+$.  Define 
	$$\mu_\cl=(\gamma_{\cl}^{<\gamma_\cl})^+.$$  
	By design, $\lambda\slesseq\gamma_\cl\slesseq\mu_\cl$.
\end{defi}

The following theorem is almost automatic, after a careful reading of the proof of \cite[3.2]{BT-R}.  We will give a general outline of the details we need from their proof, and make careful note of the places where our weaker assumption---$\mu_\cl$-measurability, rather than $\mu_\cl$-strong compactness---forces modifications to the argument.

\begin{thm}\label{thmmeaspure}
Let $\lambda$ be a regular cardinal and $\cl$ a $\lambda$-accessible category such that there exists a regular $\mu_\cl$-measurable cardinal $\kappa$, $\kappa\sgeq\mu_\cl$.  Then the $\lambda$-pure powerful image of any $\lambda$-accessible functor to $\cl$ preserving $\mu_\cl$-presentable objects is $\kappa$-preaccessible and closed under colimits of (nondegenerate) $\kappa$-chains in $\cl$.
\end{thm}

\begin{proof}
	Let $F:\ck\to\cl$ be a $\lambda$-accessible functor.  Paralleling the argument of \cite{BT-R}, we proceed in three steps:
	
	{\bf Step 1:} We can realize the $\lambda$-pure powerful image of $F$ as the full image of a functor
	$$H:\cp\to\cl$$
	where $\cp$ is the category of $\lambda$-pure morphisms $L\to FK$ in $\cl$, $K$ in $\ck$, and $H$ is the forgetful functor taking a morphism $L\to FK$ to its domain $L$.  By work of \cite{BT-R}, $\cp$ is $\mu_\cl$-accessible, and has $\lambda$-directed colimits; $H$ preserves $\lambda$-directed colimits.
	
	{\bf Step 2:} We can reformulate the situation so that we are concerned with a functor between categories of concrete structures, opening up the possibility of using ultrafilters.  
	
	We again leave the details to \cite{BT-R}.  In short, the categories $\cl$ and $\cp$ can be fully embedded in categories of many-sorted finitary structures $\Str(\Sigma_\cl)$ and $\Str(\Sigma_\cp)$, respectively, and $\cp$ can be identified with a subcategory $\Mod(T)$ of $\Str(\Sigma_\cp)$ consisting of models of a $L_{\mu_\cl,\mu_\cl}(\Sigma_\cp)$-theory $T$.  This can also be considered as a theory in signature $\Sigma=\Sigma_\cl\coprod\Sigma_\cp$, with $E:\Mod(T)\to\Str(\Sigma)$ the corresponding embedding.  Consider the reduct functor $R:\Mod(T)\to\Str(\Sigma_\cl)$ obtained by composing $E$ with the obvious reduct $\Str(\Sigma)\to\Str(\Sigma_\cl)$.  The full image of $R$, which is precisely $\Red{\Sigma}{\Sigma_\cl}{T}$, the subcategory of reducts of models of $T$ to $\Sigma_\cl$, is equivalent to the full image of $H$.  Moreover, $R$ preserves $\mu_\cl$-directed colimits and $\mu_\cl$-presentable objects.  So we can reduce all questions about the full image of $H$ to ones about $\Red{\Sigma}{\Sigma_\cl}{T}$.
	
	By assumption, $\kappa\sgeq\mu_\cl\sgeq\lambda$, so $R$ is $\kappa$-accessible and preserves $\kappa$-presentable objects: in particular, any object in $\Red{\Sigma}{\Sigma_\cl}{T}$ is a $\kappa$-directed colimit of $\kappa$-presentable $\Str(\Sigma)$-objects in $\Red{\Sigma}{\Sigma_\cl}{T}$.  That is, $\Red{\Sigma}{\Sigma_\cl}{T}$ is $\kappa$-preaccessible. We wish to show that $\Red{\Sigma}{\Sigma_\cl}{T}$ is closed in $\Str(\Sigma)$ under colimits of $\kappa$-chains.
	
	{\bf Step 3:} We proceed via an ultrafilter argument resembling that of \cite{BT-R}.  Let $D:\kappa\to\Red{\Sigma}{\Sigma_\cl}{T}$ be a $\kappa$-chain, and let $M$ be its colimit in $\Str(\Sigma)$.  To show that $M$ is in $\Red{\Sigma}{\Sigma_\cl}{T}$, it suffices to show that there is a $\lambda$-pure embedding of $M$ into an object of $\Red{\Sigma}{\Sigma_\cl}{T}$.  We capture the existence of such an embedding, as in \cite{BT-R}, by taking the $\lambda$-pure diagram $T_M$ of $M$; that is, $T_M$ consists of the positive-primitive and negated positive-primitive formulas in $L_{\lambda,\lambda}(\Sigma_M)$ satisfied by $M$.  Here $\Sigma_M$ is simply $\Sigma$ with the addition of a new constant symbol $c_m$ for each element $m\in M$.  An object $N$ of $\Str(\Sigma)$ admits a $\lambda$-pure embedding $M\to N$ just in case it satisfies $T_M$.  So it suffices to exhibit a model of $T\cup T_M$.
	
	Since we are assuming that $\kappa$ is $\mu_\cl$-measurable, we know that there is a uniform, $\mu_\cl$-complete ultrafilter $\cu$ on $\kappa$.  We claim that the ultraproduct $\prod_{\cu}M_i$ satisfies $T\cup T_M$.  As $\cu$ is $\mu_\cl$-complete, and hence $\lambda$-complete, Los's Theorem for $L_{\lambda,\lambda}$ will apply.
	
	From this point on, the required checks are precisely as in \cite{BT-R}---we include only the basic setup.  Naturally, all the $M_i$ are already models of $T$.  Expand each $M_i$ to a $\Sigma_M$-structure as follows:
	\begin{itemize}
	\item If $m$ is in the image of the colimit coprojection $M_i\to M$, interpret $c_m$ as its preimage---we call this a {\it coherent} assigment.
	\item Otherwise, assign $c_m$ an arbitary value in $M_i$.
	\end{itemize}
	By $\mu_\cl$-directedness of the chain (which follows from nondegeneracy and Lemma~\ref{cofmumeas}), and the fact that $C$ is the set-theoretic direct limit of the $M_i$, for any $\phi$ in $L_{\lambda,\lambda}(\Sigma_M)$ there is $i<\kappa$ such that all constant symbols appearing in $\phi$ are interpreted coherently in $M_i$.  Moreover, $\phi$ holds in $M$ just in case it holds in some $M_{i'}$, $i'\geq i$.  By the argument of \cite{BT-R}, a positive-primitive (or negated positive-primitive) formula $\phi$ holds in $M$ just in case it holds for all $M_j$, $j\in [i',\kappa)$.  By Lemma~\ref{uppersets}, all such upper sets belong to $\cu$, and we are done.
\end{proof}

\begin{rem}\label{pfremarks}{\em 
	\begin{enumerate}
		\item In \cite{BT-R}, the authors invoke $\kappa$-directedness to assure that any formula $\phi$ will eventually have a coherent assignment of its variables.  Here we use the weaker, but sufficient, condition of $\mu_\cl$-directedness; although we will not dwell on the possibilility, this leaves open the prospect of working with chains of singular length.
		\item Notice that the condition that $\kappa\sgeq\mu_\cl$ is used only in the proof of $\kappa$-preaccessibility of the $\lambda$-pure powerful image.  As $\mu$-strong compactness passes upward, \cite{BT-R} can simply take $\kappa\sgeq\mu_\cl$ without loss of generality.  The weakenings of measurability under consideration here do not have this property, so we need to add it as an explicit assumption if we wish to have preaccessibility.
	\end{enumerate} }
\end{rem}

In connection with Remark~\ref{pfremarks}(2), we recall Proposition~\ref{almeassharp}---if we assume $\kappa$ is {\em almost measurable}, the sharp inequality in the theorem holds automatically.  That is, we have:

\begin{cor}\label{almeaspure}
	Let $\lambda$ be a regular cardinal and $\cl$ a $\lambda$-accessible category such that there exists an almost measurable cardinal $\kappa>\mu_\cl$.  Then the $\lambda$-pure powerful image of any $\lambda$-accessible functor to $\cl$ preserving $\mu_\cl$-presentable objects is $\kappa$-preaccessible and closed under colimits of $\kappa$-chains in $\cl$.
\end{cor}

Moving in another direction, we could remain with the weaker assumption of $\mu$-measurability but drop the sharp inequality, at the cost of preaccessibility.

\begin{propo}\label{measpurechain}
	Let $\lambda$ be a regular cardinal and $\cl$ a $\lambda$-accessible category such that there exists a regular $\mu_\cl$-measurable cardinal $\kappa$.  Then the $\lambda$-pure powerful image of any $\lambda$-accessible functor to $\cl$ preserving $\mu_\cl$-presentable objects closed under colimits of $\kappa$-chains in $\cl$.
\end{propo}

As noted in \cite[3.4]{BT-R}, we can run the argument for Theorem~\ref{thmmeaspure} with monomorphisms in place of $\lambda$-pure monomorphisms, taking $T_M$ in Step 3 to be the atomic (and negated atomic) formulas true in $M$.  This gives

\begin{thm}\label{thmmeasmono}
	Let $\lambda$ be a regular cardinal and $\cl$ a $\lambda$-accessible category such that there exists a regular $\mu_\cl$-measurable cardinal $\kappa$, $\kappa\sgeq\mu_\cl$.  Then the powerful image of any $\lambda$-accessible functor to $\cl$ preserving $\mu_\cl$-presentable objects is $\kappa$-preaccessible and closed under colimits of $\kappa$-chains in $\cl$.
\end{thm}

As above, this proof generates a pair of additional consequences:

\begin{propo}\label{almeasmono}
	Let $\lambda$ be a regular cardinal and $\cl$ a $\lambda$-accessible category such that there exists an almost measurable cardinal $\kappa>\mu_\cl$.  Then the powerful image of any $\lambda$-accessible functor to $\cl$ preserving $\mu_\cl$-presentable objects is $\kappa$-preaccessible and closed under colimits of $\kappa$-chains in $\cl$.
\end{propo}

\begin{propo}\label{measmono}
	Let $\lambda$ be a regular cardinal and $\cl$ a $\lambda$-accessible category such that there exists a regular $\mu_\cl$-measurable cardinal $\kappa$.  Then the powerful image of any $\lambda$-accessible functor to $\cl$ preserving $\mu_\cl$-presentable objects is closed under colimits of $\kappa$-chains in $\cl$.
\end{propo}

\section{Locality and chain closure}\label{seclocal}

We now turn to an examination of locality properties of Galois types, and must therefore introduce concreteness to the picture.  We restrict our attention to a category-theoretic framework slightly more general than the $\mu$-AECs of \cite{mu-aec-jpaa}, hence also more general than AECs.  We note that we also work slightly more generally than \cite{LRclass}, which is concerned with accessible categories with concrete directed colimits.

\begin{defi}
	By a {\em concrete $\lambda$-accessible category}, we mean a pair $(\ck,U)$, where $\ck$ is a $\lambda$-accessible category and $U:\ck\to\Set$ is a faithful functor that preserves $\lambda$-directed colimits.  We require, too, that $U$ preserve $\lambda$-presentable objects, i.e. if $M$ is $\lambda$-presentable, $|UM|<\lambda$.
\end{defi}

We will abuse this terminology slightly: in general, we will simply refer to $\ck$ itself as an concrete accessible category, leaving the forgetful functor $U$ implicit.  

\begin{remark} {\em By Fact~\ref{aecmuaeccats}, any $\mu$-AEC with L\"owenheim-Skolem-Tarski number $\lambda$ is a concrete $\lambda^+$-accessible category with concrete $\mu$-directed colimits.  As before, AECs correspond to the case $\mu=\aleph_0$.}
\end{remark}

As in \cite{LRclass}, we define Galois-type-theoretic notions in ways suited to the context of concrete accessible categories.  It will be abundantly clear, though, that they specialize to precisely the appropriate notions in, say, AECs.

\begin{defi}
	Let $\ck$ be a concrete accessible category, and let $M$ be an object of $\ck$.  A {\em Galois type} over $M$ is an equivalence class of pairs $(f,a)$, $f:M\to N$ and $a\in UN$, where we say that 
	$(f_0,a_0)\sim (f_1,a_1)$
	for $f_i:M\to N_i$ and $a_i\in UN_i$ just in case there is an object $N$ and morphisms $g_i:N_i\to N$ such that the following diagram commutes
	$$  
      \xymatrix@=3pc{
        N_0 \ar[r]^{g_0} & N \\
        M \ar [u]^{f_0} \ar [r]_{f_1} &
        N_1 \ar[u]_{g_1}
      }
      $$
      with $U(g_0)(a_0)=U(g_1)(a_1)$.
\end{defi}

We note that this need not be an equivalence relation in general, as transitivity may fail.  It does hold if we assume that $\ck$ has the {\em amalgamation property} (or {\em AP}); that is, any span $N_0\leftarrow M\to N_1$ in $\ck$ can be completed to a commutative square.  In the interest of transparency, we explicitly include this hypothesis in the results that follow: please note that the AP is also in the background of \cite{LRclass} and \cite{BT-R}, and the combinatorial AEC constructed in \cite{boun} has amalgamation, so there is not, in fact, any difference in the scope of our result.

\begin{defi}
	We say that Galois types in $\ck$ are {\em $\kappa$-local} if for any object $M$, any continuous $\kappa$-chain
	$$M_0\to M_1\to\dots\to M_i\to\dots$$
	with colimit $M$ and colimit coprojections $(\phi_i:M_i\to M)$, and any pair $(f_0:M\to N_0,a_0)$ and $(f_1:M\to N_1,a_1)$, if
	$$(\phi_i f_0,a_0)\sim (\phi_i f_1,a_1)$$ 
	for all $i<\kappa$, then
	$$(f_0,a_0)\sim (f_1,a_1)$$
\end{defi}

As in \cite{LRclass} or \cite{LiRo17}, the game is to transform $\kappa$-locality into a problem involving powerful images.  In fact, the category-theoretic formulation we use here is precisely the same, although we make one additional observation, Remark~\ref{typecatsrmk}(2) below, that is crucial to obtaining a one-cardinal version of our central result.

\begin{nota}\label{typecatsdef}
Let $\ck$ be a concrete $\lambda$-accessible category with amalgamation.
\begin{enumerate}
	\item Let $\cl_1$ be the category of diagrams witnessing equivalence of pairs in $\ck$, i.e. all commutative squares of the form 
	$$  
      \xymatrix@=3pc{
        N_0 \ar[r]^{g_0} & N \\
        M \ar [u]^{f_0} \ar [r]_{f_1} &
        N_1 \ar[u]_{g_1}
      }
    $$
    with selected elements $a_i\in UN_i$, $i=0,1$, and $U(g_0)(a_0)=U(g_1)(a_1)$.  The morphisms are, of course, morphisms of squares that preserve the selected elements.
   	\item Let $\cl_2$ be the category of pairs of types, i.e. the category of pointed spans of the form
   	$$  
      \xymatrix@=3pc{
        N_0  &  \\
        M \ar [u]^{f_0} \ar [r]_{f_1} &
        N_1 
      }
    $$
    with selected elements $a_i\in UN_i$, $i=0,1$.
    \item Let $F$ be the obvious forgetful functor from $\cl_1$ to $\cl_2$.
\end{enumerate}
\end{nota}

\begin{remark}\label{typecatsrmk} {\em 
	\begin{enumerate}
		\item As noted in the proof of \cite[5.2]{LRclass}, the full image of $F$ in $\cl_2$ consists precisely of the $\sim$-equivalent pairs $(f_i:M\to N_i,a_i\in UN_i)$, $i=0,1$.  The full image is clearly closed under subobjects, hence powerful.
		\item We recall from the same proof that $\cl_1$, $\cl_2$, and $F$ are all accessible.  In fact, one can see that $F$ is $\lambda$-accessible and preserves $\nu$-presentable objects for all $\nu\sgreat\lambda$---including $\mu_{\cl_2}$---and that $\mu_{\cl_2}=\mu_\ck$.
	\end{enumerate} }
\end{remark}

\begin{theo}\label{powfimlocal}
	Let $\ck$ be a concrete $\lambda$-accessible category with amalgamation.  If a cardinal $\kappa>\mu_\ck$ has the property that the powerful image of any $\lambda$-accessible functor $F:\cm\to\cl$ whose codomain satisfies $\mu_\cl<\kappa$ and that preserves $\mu_\cl$-presentable objects is closed under colimits of $\kappa$-chains, types in $\ck$ are $\kappa$-local.
\end{theo}

\begin{proof}
	By Remark~\ref{typecatsrmk}(2), the powerful image of $F:\cl_1\to\cl_2$ is closed under colimits of $\kappa$-chains.  By Remark~\ref{typecatsrmk}(1), the powerful image is precisely the category of equivalent pairs.
	
	Let $M$ be an object of $\ck$, and suppose that $M$ is the colimit of a continuous $\kappa$-chain
	$$M_0\to M_1\to\dots\to M_i\to\dots$$
	with colimit coprojections $\phi_i:M_i\to M$, $i<\kappa$.  Consider a pair of (representatives of) types
	$$
	\xymatrix@=1pc{
         & N_0 & (a_0\in UN_0)\\
        M \ar [ur]^{f_0} \ar [dr]_{f_1} & & \\
        & N_1 & (a_1\in UN_1)
      }
    $$
	Suppose that they are equivalent over any $M_i$; that is, for any $i<\kappa$, there are $g_0^i:N_0\to N^i$ and $g_1^i:N_1\to N^i$ such that we have the following diagram
	$$
	\xymatrix@=2pc{
         &&&&&& N_0\ar[dr]^{g_0^i} & \\
       M_0\ar[r] & M_1\ar[r] & \dots\ar[r] & M_i\ar[r]\ar@/^/[urrr]^{f_0\phi_i}\ar@/_/[drrr]_{f_1\phi_i} & \dots\ar[r] &  M \ar [ur]_{f_0} \ar [dr]^{f_1} & & N^i\\
        &&&&&& N_1\ar[ur]_{g_1^i} & 
      }
    $$
    with $g_0^if_0\phi_i=g_1^if_1\phi_i$ and $U(g_0)(a_0)=U(g_1)(a_1)$.  In other words,
    $$
	\xymatrix@=1pc{
         & N_0 & (a_0\in UN_0)\\
        M_i \ar [ur]^{f_0\phi_i} \ar [dr]_{f_1\phi_1} & & \\
        & N_1 & (a_1\in UN_1)
      }
    $$
    belongs to the powerful image of $F$ for all $i<\kappa$.  Since the powerful image of $F$ is closed under colimits of $\kappa$-chains (see Corollary~\ref{measmono}), and since the original pair of types is clearly the colimit of the $\kappa$-chain of pairs of the form immediately above, that original pair must belong to the powerful image as well.  That is, $(f_0,a_0)$ and $(f_1,a_1)$ are equivalent.
\end{proof}

As a special case:

\begin{thm}\label{thmmuaec}
	Let $\ck$ be a $\nu$-AEC with amalgamation, $\lsnk=\lambda$.  If there is a cardinal $\kappa>\mu_\ck$ such that the powerful image of any $\lambda^+$-accessible functor $F:\cm\to\cl$ whose codomain satisfies $\mu_\cl<\kappa$ and preserves $\mu_\cl$-presentable objects is closed under colimits of $\kappa$-chains, then $\ck$ is $\kappa$-local.
\end{thm}

In particular, this holds for AECs, which correspond to the case $\nu=\aleph_0$.

\section{Chain closure as a large cardinal property} 

We now close the chain of implications.  We require the following result of \cite{boun}:

\begin{thm}[\cite{boun}, 4.9(1)]\label{bounsigmaplus} Let $\lambda<\kappa$ be infinite cardinals with $\lambda^{\omega}=\lambda$.  If every AEC $\ck$ with amalgamation and $\lsnk=\lambda$ is $\kappa$-local, $\kappa$ is $\lambda^+$-measurable.
\end{thm}

In the style of \cite[4.10]{boun}, this yields the following:

\begin{cor}\label{bounimplicit}
	Let $\kappa$ be an infinite cardinal with $\mu^{\omega}<\kappa$ for all $\mu<\kappa$.  If every AEC $\ck$ with amalgamation and $\lsnk<\kappa$ is $\kappa$-local, $\kappa$ is almost measurable.
\end{cor}

This, together with the work of previous sections, gives us the main result:

\begin{thm}\label{almeasequiv}
	Let $\kappa$ be a strong limit cardinal.  The following are equivalent:
	\begin{enumerate}
		\item $\kappa$ is almost measurable.
		\item Let $\lambda$ be a regular cardinal.  For any $\lambda$-accessible functor $F:\ck\to\cl$ with $\mu_\cl<\kappa$ that preserves $\mu_\cl$-presentable objects, the powerful image of $F$ is $\kappa$-preaccessible and closed under colimits of $\kappa$-chains in $\cl$.
		\item Any ($\nu$-)AEC $\ck$ with amalgamation and $\lsnk<\kappa$ is $\kappa$-local.
	\end{enumerate}
\end{thm}

\begin{proof}
	$(1)\Rightarrow (2)$: Corollary~\ref{almeasmono}.
	
	$(2)\Rightarrow (3)$: Note that if $\lsnk<\kappa$, $\mu_\ck<\kappa$ by Fact~\ref{aecmuaeccats}(2) and Definition~\ref{sharpdef}.  So Theorem~\ref{thmmuaec} applies.
	
	$(3)\Rightarrow (1)$: Since $\kappa$ is a strong limit cardinal, $\mu^\omega<\kappa$ for all $\mu<\kappa$, so we may use Corollary~\ref{bounimplicit}.
\end{proof}

\begin{remark}  {\em 
\begin{enumerate}
\item We note that the implication $(1)\Rightarrow (3)$ follows from \cite[5.2]{boneylarge}.	
\item The equivalence of Theorem~\ref{almeasequiv} fits a still-emerging picture of (possibly weakened) measurability as a kind of chain completeness or compactness property.  A noteworthy early variation on this theme appears as Exercise 4.26 in \cite{changkeis}, and was recently resurrected as \cite[Fact~1.2]{boneymodlcs}: a cardinal $\kappa$ is measurable just in case every $L_{\kappa,\kappa}$-theory that can be written as the union of an increasing sequence of satisfiable theories is  itself satisfiable.
\end{enumerate} }
\end{remark}

It is tempting to try to produce an analogous version of Theorem~\ref{almeasequiv} involving $\mu$-measurability rather than almost measurability, using Theorem~\ref{bounsigmaplus} in place of Corollary~\ref{bounimplicit}. The difficulty lies in the gap between $\mu_\cl$ and $\lsnk^+$, which are, respectively, the degree of measurability needed to begin the chain of implications, and the degree of measurability that we can infer from locality---in any AEC $\ck$, $\mu_\ck$ will be considerably larger than $\lsnk^+$.  This gap can likely be narrowed, if not closed: in this account, we have used far less than the full structure of an AEC (or $\mu$-AEC).  The fact that all morphisms are (concrete) monomorphisms alone should yield a much simpler result along the lines of Theorem~\ref{thmmeasmono}, with a smaller cardinal in place of $\mu_\cl$.  This may be all that is required.

\bibliographystyle{alpha}
\bibliography{measurable-accessible.bib}

\end{document}